\documentclass[preprint]{elsarticle}
\usepackage{graphicx, epsfig} %include figure files
\usepackage{amsmath, amsfonts, amsthm}
\usepackage{bm} %include bold math: \bm{} creates bold letters in math mode
\usepackage{pstricks,psfrag,pst-node,pst-text,pst-3d,pst-grad}

\usepackage[margin=2.5cm]{geometry}
\usepackage[linktocpage]{hyperref}
\usepackage[caption=false]{subfig}
\usepackage{diagbox}

\newcommand{\del}{\partial}

\newcommand{\ve}{\bm{e}}
\newcommand{\vp}{\bm{p}}
\newcommand{\vq}{\bm{q}}
\newcommand{\vu}{\bm{u}}
\newcommand{\vv}{\bm{v}}
\newcommand{\vw}{\bm{w}}
\newcommand{\vx}{\bm{x}}
\newcommand{\vF}{\bm{F}}
\newcommand{\vzero}{\bm{0}}

\newcommand{\bx}{\bar{x}}

\newtheorem{thm}{Theorem}
\newtheorem{lem}{Lemma}
\newtheorem{prop}{Proposition}

\newtheorem{defn}{Definition}

\begin{document}
\title{Singularity-Avoiding Multi-Dimensional Root-Finder}

\author{Hirotada Okawa}\ead{h.okawa@aoni.waseda.jp}
\address{Waseda Institute for Advanced Study, Waseda University, Tokyo 169-0051, Japan}
\author{Kotaro Fujisawa}
\address{Department of Physics, Graduate School of
Science, University of Tokyo, Bunkyo-ku, Tokyo 113-0033, Japan}
\author{Yu Yamamoto}
\address{Research Institute for Science and Engineering, Waseda University, Tokyo 169-8555, Japan}
\author{Nobutoshi Yasutake}
\address{Physics Department, Chiba Institute of Technology, Chiba 275-0023, Japan}
\address{Advanced Science Research Center, Japan Atomic Energy Agency, Tokai, Ibaraki 319-1195, Japan}
\author{Misa Ogata}
\address{Research Institute for Science and Engineering, Waseda University, Tokyo 169-8555, Japan}
\author{Shoichi Yamada}
\address{Research Institute for Science and Engineering, Waseda University, Tokyo 169-8555, Japan.}
% \address{Science and Engineering, Waseda University, Tokyo 169-8555, Japan.}

\date{\today} 

\begin{abstract}
 We proposed in this paper a new
 method, which we named the W4 method, to solve nonlinear equation systems.
 It may be regarded as an extension of the Newton-Raphson~(NR) method
 to be used when the method fails.
 Indeed our method can be applied not only to ordinary problems
 with non-singular Jacobian matrices
 but also to problems with singular Jacobians,
 which essentially all previous methods that employ
 the inversion of the Jacobian matrix have failed to solve.
 In this article,
 we demonstrate that (i) our new scheme
 can define a non-singular iteration map even for those problems by
 utilizing the singular value decomposition,
 (ii) a series of vectors in the new iteration map converges to the
 right solution under a certain condition,
 (iii) the standard two-dimensional problems in the
 literature that no single method proposed so far has been able to solve completely
 are all solved by our new method.
\end{abstract}

% \pacs{
% 02.60.Cb,% Numerical simulation; solution of equations 
% 02.70.−,% Computational techniques; simulations 
% 02.90.+p,% Other topics in mathematical methods in physics
% }
\maketitle

%%%%%%%%%%%%%%%%%%%%%%%%%%%%%%%%%%%%%%%%%%%%%%%%%%%%%%%%%%%%%%%%%%%%%%%%%%%%
\section{Introduction}\label{sec:intro}
%%%%%%%%%%%%%%%%%%%%%%%%%%%%%%%%%%%%%%%%%%%%%%%%%%%%%%%%%%%%%%%%%%%%%%%%%%%%
%
The root-finding of functions is one of the most important problems
in computational science and engineering.
It is a nontrivial task, however, to numerically find the root of a system of
nonlinear equations:
\begin{eqnarray}
 \vF(\vx) = \mathbf{0},\label{eq:nonlinearEQs}
\end{eqnarray}
where $\vx \in \mathbb{R}^{N}$, $N\in \mathbb{Z}$
and $\bm{F}: \mathbb{R}^{N}\rightarrow\mathbb{R}^{N}$.
Although the single-variable problem is rather simple,
it becomes particularly difficult when $N>1$.
The Newton-Raphson~(NR) method may be the first choice,
since it is well-known to give a solution
as long as the initial guess is sufficiently close to the
solution\cite{ortega1970,kelley2003}.
Another advantage for the NR method is its quadratic convergence to the
solution if the Jacobian matrix for the system of nonlinear
equations~\eqref{eq:nonlinearEQs} is non-singular.
Many attempts done so far to further accelerate the convergence: 
Halley's and Householder's methods
for single-variable problems\cite{householder1970}
and Ramos\& Monteiro's method for multi-variable problems to mention a few~\cite{ramos2015,ramos2017}.

There are demerits in the NR method, though.
Very heavy computational
cost in the inversion of Jacobian matrix for large system dimensions is
one of them: the operation number
 scales as $\mathcal{O}\left(N^3\right)$.
Many efforts have been successfully made to reduce
the cost to $\mathcal{O}\left(N^2\right)$\cite{broyden1965,kelley2003}.
Another common disadvantage for the NR and other quasi-Newton methods
is the strong dependence of convergence on the initial condition for iteration.
It is crucial indeed for
whether the iteration can reach a solution or not.
We commonly come across situations, in which the iteration is simply
divergent or suffers from permanent oscillations.
Recently, we proposed a new method referred to as the W4 method to
circumvent such difficulties.
We have shown
that the W4 method is able to obtain a solution even if the NR method fails\cite{Okawa:2018smx}.
In fact, the W4 method has been successfully applied
to some physical problems
already\cite{Fujisawa:2018dnh,Suzuki:2020zbg,Hirai:2020sjg}.

The existence of singularity in the Jacobian matrix
is a different issue.
If the Jacobian matrix is singular at the solution,
the convergence is slowed down severely and
we are required to take some measure to reaccelerate it\cite{kou2006,wu2007,hueso2009}.
If the Jacobian is singular either at the initial guess or at
intermediate steps in the iteration, the problem is much more serious,
since one cannot invert the Jacobian matrix and hence cannot define the iteration map to the
next step.
In the existing multi-variable root finders, to the best of our
knowledge,
one needs to somehow modify (normally by hand) the initial guess or the
intermediate results in such situations.
In this article, we try to deal with this difficulty in the framework of
the W4 method and establish
the foundation of a globally convergent multi-variable root-finder.
We hence focus on two dimensional problems with singular Jacobian
matrices in this article.

%%% Contents
The paper is organized as follows.
We first present our new scheme to find roots of nonlinear
equation systems in the framework of the W4 method in Sec.~\ref{sec:w4sv},
clarifying how the singular nature of the associated
Jacobian matrix is handled.
We also show that our method is applicable to 
ordinary non-singular problems as well.
In Sec.~\ref{sec:results}, we demonstrate the capabilities of the new method
by applying it to the standard test problems in the literature
comparing the results with those of other methods.
Finally, we summarize our findings and comment on future prospects in Sec.~\ref{sec:conclusion}.

%%%%%%%%%%%%%%%%%%%%%%%%%%%%%%%%%%%%%%%%%%%%%%%%%%%%%%%%%%%%%%%%%%%%%%%%%%%%
\section{Singularity Avoidance with the W4 method}\label{sec:w4sv}
%%%%%%%%%%%%%%%%%%%%%%%%%%%%%%%%%%%%%%%%%%%%%%%%%%%%%%%%%%%%%%%%%%%%%%%%%%%%
In this paper, we consider two dimensional nonlinear equations,
which are expressed in general as
\begin{eqnarray}
 \vF (x,y) \equiv
 \begin{pmatrix}
  f_{x}(x,y)\\
  f_{y}(x,y)
 \end{pmatrix}
 =
 \vzero .
\end{eqnarray}
For iterative solvers, such as 
the Newton-Raphson method and the W4 method, 
it is normally necessary to calculate the Jacobian matrix associated with the
system of equations\footnote{Some quasi-Newton methods do not
require the inversion of Jacobian matrix explicitly.
It is essential
for numerical stability, however,
to make the algorithm as close to the original inversion as possible.\cite{broyden1965}}:
\begin{eqnarray}
 J =
  \begin{pmatrix}
   \frac{\del f_x}{\del x} & \frac{\del f_x}{\del y}\\
   \frac{\del f_y}{\del x} & \frac{\del f_y}{\del y}
  \end{pmatrix}.
\end{eqnarray}
We call the Jacobian singular when $\det J=0$.
In our previous article~\cite{Okawa:2018smx},
we proposed a new multi-dimensional
root-finding scheme, the W4 method, and demonstrated that
it can solve some problems that
the Newton-Raphson method fails to solve.
As shown, the W4 method
with the UL~\cite{Okawa:2018smx} or
the LH decomposition~\cite{Fujisawa:2018dnh} has
a tendency to leap over singularities 
by inertia.
We observed, however, that even the W4 method is stalled sometimes,
particularly when the initial guess has a singular Jacobian.
The new method we propose in this article solves this problem.
Below we explain how we define the iteration map in this method for such situations, i.e., when a
singularity is encountered either at the initial step or 
at some intermediate steps in the iteration.
\subsection{Eigendecomposition} \label{sec:eigendecomp}
At first, we look at some relevant properties of the
 Jacobian matrix that may be singular.
Suppose we have a $2\times 2$ real matrix~$A$.
The two eigenvalues for this matrix~$A$ are given as
\begin{eqnarray}
 \lambda_{\pm} = \frac{1}{2}\left(\alpha\pm\sqrt{\alpha^2-4\beta}\right),\label{eq:lambdapm}
\end{eqnarray}
where we define $\alpha:=\mathrm{tr} A$ and $\beta:=\det A$.
Assuming eigenvectors~$\vw_{\pm}$ for $\lambda_{\pm}$,
respectively,
we decompose the Jacobian into
\begin{eqnarray}
 A = P \Lambda P^{-1} = 
  \begin{pmatrix}
   \vw_{+} & \vw_{-} 
  \end{pmatrix}
  \begin{pmatrix}
   \lambda_{+}& 0\\
   0 & \lambda_{-}
  \end{pmatrix}
  \begin{pmatrix}
   \vw_{+}^{T}\\
   \vw_{-}^{T}
  \end{pmatrix}.
\end{eqnarray}
Note that the eigenvalues are generally complex unless the Jacobian is symmetric.

For a singular Jacobian~$J$,
one eigenvalue is $\lambda_{+}=\alpha$ and the other is $\lambda_{-}=0$
since $\beta=\det J=0$ by definition from Eq.~\eqref{eq:lambdapm}.
Then the singular Jacobian can be rewritten as
\begin{eqnarray}
 J = \lambda_{+} \vw_{+}\vw_{+}^{T}.
\end{eqnarray}
\subsection{Singular Value Decomposition} \label{sec:SVD}
As mentioned above, there may not exist all the eigenvectors.
We hence consider the singular value decomposition 
of the Jacobian.  The following equations:
\begin{eqnarray}
 J\vv = \sigma \vu,\\
 J^{T}\vu = \sigma \vv,
\end{eqnarray}
yield
\begin{eqnarray}
 J^{T}J\vv = \sigma^2\vv,\\
 JJ^{T}\vu = \sigma^2\vu.
\end{eqnarray}
In these equations, $\sigma$'s are called 
the singular values of $J$. Since they are
defined as
the positive square root of $\sigma^2$,
the eigenvalues of positive-semidefinite symmetric matrix
of $J^{T}J$ or $JJ^{T}$, the corresponding vectors, $\vu$ and $\vv$,
always exist.
\begin{lem}
 Suppose a $2\times 2$ Jacobian matrix is singular, then
 one of the singular values is at least zero.
\end{lem}
\begin{proof}
 When the Jacobian~$J$ is singular, i.e., $\det J=0$,
 we obtain $\beta = \det A = \det (J^TJ) = (\det J)^2 =0$
 for the matrix~$A=J^TJ$.
Then from Eq.~\eqref{eq:lambdapm},
 the square of the smaller singular
 value is given as $\sigma^2_{-}=(\alpha-\sqrt{\alpha^2 -4\beta})/2=0$.
\end{proof}
Let $U=[\vu_{+}\ \vu_{-}]$ and $V=[\vv_{+}\ \vv_{-}]$
be the orthogonal matrices given by the two independent eigenvectors
associated with $JJ^T$ and $J^TJ$, respectively,
and $\Sigma=\mathrm{diag} [\sigma_{+},\ \sigma_{-}]$ be the diagonal matrix.
The Jacobian can be decomposed as
\begin{eqnarray}
 J = U\Sigma V^{-1} =
  \begin{pmatrix}
   \vu_{+}& \vu_{-}
  \end{pmatrix}
  \begin{pmatrix}
   \sigma_{+} & 0\\
   0 & \sigma_{-}
  \end{pmatrix}
  \begin{pmatrix}
   \vv_{+}^{T}\\
   \vv_{-}^{T}
  \end{pmatrix},
  \label{eq:Jsinglar}
\end{eqnarray}
which is the singular value decomposition of $J$.
\begin{lem}
 Suppose a $2\times 2$ Jacobian is singular,
 then it can be expressed only by the right-singular and
 left-singular vectors, $\vv_{+}$ and $\vu_{+}$, which correspond
to the
 larger singular value~$\sigma_{+}$.
\end{lem}
\begin{proof}
 Since $\sigma_{-}=0$,
 Eq.~\eqref{eq:Jsinglar} yields $J=\sigma_{+}\vu_{+}\vv_{+}^T$.
\end{proof}
Note that, in general, the singular values~$\sigma$ differ from
the eigenvalues~$\lambda$
and they coincide with each other 
if the Jacobian is symmetric~$J=J^{T}$.

\subsection{W4 with singular value decomposition}\label{ssec:W4}
The generic form of
 the iteration map of the W4 method for variable $\vx^{(n)}$
and auxiliary variable~$\vp^{(n)}$
at the $n$-step is given as
\begin{subequations}
 \label{eq:W4map}
 \begin{align}
  \bm{x}^{(n+1)} &= \bm{x}^{(n)} +\Delta\tau X \bm{p}^{(n)},
  \label{eq:W4map_x}\\
  \bm{p}^{(n+1)} &= \left(1-2\Delta\tau\right)\bm{p}^{(n)}
  -\Delta\tau Y\bm{F}(\bm{x}^{(n)}),
  \label{eq:W4map_p}
 \end{align}
\end{subequations}
where $X$ and $Y$ are preconditioner matrices,
which we can choose at our disposal in principle.
Linearizing the above nonlinear map at
the solution~($\bm{x}=\bm{x}^{*}$ and $\bm{p}=\bm{0})$
and introducing the errors as
$\bm{e}_{x}^{(n)}:=\bm{x}^{*}-\bm{x}^{(n)}$
and $\bm{e}_{p}^{(n)}:=\vzero-\bm{p}^{(n)}$,
we obtain the error propagation equations:
\begin{subequations}
 \label{eq:W4map_linear}
  \begin{align}
   \bm{e}_{x}^{(n+1)} =& \bm{e}_{x}^{(n)} - X \Delta\tau \bm{e}_{p}^{(n)},\\
   \bm{e}_{p}^{(n+1)} =& \left(1 -2\Delta\tau\right)\bm{e}_{p}^{(n)}
   +YJ \Delta \tau\bm{e}_{x}^{(n)},
  \end{align}
\end{subequations}
which are rewritten as
\begin{eqnarray}
 \bm{e}_{z}^{(n+1)} = W \bm{e}_{z}^{(n)},\quad
 \bm{e}_{z}^{(n)} :=
 \begin{pmatrix}
  \bm{e}_{x}^{(n)}\\
  \bm{e}_{p}^{(n)}
 \end{pmatrix}
 ,\quad
 W :=
 \begin{bmatrix}
  I & -\Delta\tau X\\
  \Delta\tau Y J & (1-2\Delta\tau)I
 \end{bmatrix}
 ,\label{eq:W4_matrix}
\end{eqnarray}
where $I$ denotes the $N\times N$ identity matrix, with
$N$ being the number of nonlinear equations.
\begin{lem}
 Suppose there exists a complete set of
 eigenvectors~$\vv_i\in\mathbb{R}^{2N}$
 of the matrix~$W$ in Eq.~\eqref{eq:W4_matrix}
 and let $Q$ be a $2N\times 2N$ matrix composed of
 $\vv_i$ and $d_i$ be the eigenvalues corresponding to $\vv_i$,
 then the norm of the error vector
 always decreases if
 $|d_{max}|<1$ for the maximum eigenvector $d_{max}$.\label{lem:error}
\end{lem}
\begin{proof}
 The matrix~$W$ can be decomposed as $W=Q^{-1}DQ$ in terms of the matrix
 $Q:=[\vv_1 \vv_2 \cdots \vv_{2N}]$ and the diagonal matrix
 $D:=\mathrm{diag}[d_1,d_2,\cdots,d_{2N}]$.
 Then the error propagates from the $n$-step to the $(n+1)$-step as
 \begin{eqnarray}
  \ve^{(n+1)}_{z} = W\ve^{(n)}_{z}
    = Q^{-1}DQ\ve^{(n)}_{z}.
 \end{eqnarray}
 Defining the auxiliary vector~$\vq:=Q\ve^{(n)}_z$
 for notational convenience,
 we evaluate the norm of the error at the $(n+1)$-step:
 \begin{eqnarray}
  \mid \ve^{(n+1)}_{z} \mid^2
   &=& \left(Q^{-1}DQ\ve^{(n)}_{z}\right)^T Q^{-1}DQ\ve^{(n)}_{z}
   = \vq^{T}D^2\vq
   = \sum_{i=1}^{2N}d_i^2q_i^2\nonumber\\
   &<& \sum_{i=1}^{2N} q_i^2 = \mid\vq\mid^2 = \mid Q\ve^{(n)}_{z} \mid^2
   = \mid \ve^{(n)}_{z} \mid^2,
 \end{eqnarray}
where we used $|d_i|<1$ to derive the inequality
and $Q^{T}=Q^{-1}$ to obtain the last equality.
\end{proof}
In the framework o the W4 method,
the eigenvalues of the matrix $W$ in Eq.~\eqref{eq:W4_matrix}
are cruicially important,
which are obtained from the characteristic polynomial:
\begin{eqnarray}
 \det\left[W -d_WI\right] = \det\left[(1-d_W)(1-2\Delta\tau -d_W)I
      +\Delta\tau^2 Y J X \right]=0. \label{eq:eigenW4map0}
\end{eqnarray}

\begin{defn}
 We define the W4SV method as the following choice of matrices
 $X$ and $Y$ in Eqs.~\eqref{eq:W4map_x} and \eqref{eq:W4map_p}:
 $X=V$ and
 $Y=\hat\Sigma^{-1}U^{-1}$,
 where
 $U, V$ and $\hat\Sigma$ are the matrices in the singular
 decomposition of Jacobian $J$ and we define 
 $\hat\Sigma^{-1}$ as
 \begin{eqnarray}
  \hat\Sigma^{-1} =
    \begin{cases}
     \mathrm{diag}\left[\sigma_{+}^{-1},\sigma_{-}^{-1}\right]
     &     (\sigma_{-}\neq 0),\\
     \mathrm{diag}\left[\sigma_{+}^{-1},1\right] & (\sigma_{-} = 0).
    \end{cases}
    \label{eq:hatSigma}
 \end{eqnarray}
\end{defn}

\begin{prop}
 Suppose all singular values of Jacobian~$J$ in the two dimensional
 problems are nonzero,
 then the W4SV map with $0<\Delta\tau< 1$ yields
 a series of vectors~$\bm{x}^{(n)}$ that converge to a solution
 if the initial condition is sufficiently close
 to the solution.\label{prop:w4sv_nonsingular}
\end{prop}
\begin{proof}
 Since the Jacobian matrix
 is decomposed as~$J=U\Sigma V^{-1}$,
 we have $\hat\Sigma^{-1}U^{-1}JV = I$
 
 when all the singular values are nonvanishing.
 Then the eigenvalues of the matrix~$W$ in Eq.~\eqref{eq:W4_matrix} can be calculated from Eq.~\eqref{eq:eigenW4map0}:
 \begin{eqnarray}
  \det\left[W -d_WI\right]
   =
   \det\left[(1-d_W)(1-2\Delta\tau -d_W)I
	+\Delta\tau^2 I \right]=
   \det\left[\left( 1 -\Delta\tau -d_W \right)^2 I \right]=0
 \end{eqnarray}
as 
 $d_W=1-\Delta\tau$.
It is obvious that $|d_W|<1$ for $0<\Delta\tau<1$.
If the initial condition is sufficiently close to the solution,
the linearized equation~\eqref{eq:W4_matrix}
is valid and the error should decrease monotonically
and the iteration will converge to the solution.
\end{proof}
\begin{lem}
 Suppose one of the singular values of a $2\times 2$ 
singular Jacobian matrix is vanishing,
 then one of the eigenvalues of the matrix~$W$ for the W4SV map
 is unity.
 \label{lem:w4sv_singular}
\end{lem}
\begin{proof}
 Since the singular Jacobian is written as %by
 $J=\sigma_{+}\vu_{+}\vv_{+}^T$,
 the matrix $YJX$ is calculated 
as follows\footnote{In practice, it is helpful to relax 
the condition~$\sigma_{-}=0$ in Eq.~\eqref{eq:hatSigma} 
to $\sigma_{-}<10^{-15}$, for example.}:
 \begin{eqnarray}
  YJX &=& \hat\Sigma^{-1} U^{-1} J V =
   \begin{pmatrix}
    \sigma_{+}^{-1} & 0\\
    0 & 1
   \end{pmatrix}
   \begin{pmatrix}
    \vu_{+}^{T}\\
    \vu_{-}^{T}
   \end{pmatrix}
   \sigma_{+}\vu_{+}\vv_{+}^{T}
   \begin{pmatrix}
    \vv_{+} &
    \vv_{-}
   \end{pmatrix}\nonumber\\
  &=&
   \begin{pmatrix}
    \sigma_{+}^{-1} & 0\\
    0 & 1
   \end{pmatrix}
   \begin{pmatrix}
    \sigma_{+}\vv_{+}^{T}\\
    \vzero^T
   \end{pmatrix}
   \begin{pmatrix}
    \vv_{+} & \vv_{-}
   \end{pmatrix}
   =
   \begin{pmatrix}
    \sigma_{+}^{-1} & 0\\
    0 & 1
   \end{pmatrix}
   \begin{pmatrix}
    \sigma_{+} & 0\\
    0 & 0
   \end{pmatrix}
   =
   \begin{pmatrix}
    1 & 0\\
    0 & 0
   \end{pmatrix}.
 \end{eqnarray}
 Then the characteristic polynomial equation becomes
 \begin{eqnarray}
  0 = \det\left[W -d_W I\right] = \det
   \begin{bmatrix}
    (1-\Delta \tau -d_W)^2 & 0\\
    0 & (1-d_W)(1-\Delta \tau -d_W)
   \end{bmatrix}
   = (1-d_W)(1-\Delta \tau -d_W)^3.
 \end{eqnarray}
It is now apparent that
 one of the eigenvalues is $d_W=1$ and the other is $d_W=1-\Delta \tau$.
\end{proof}
\begin{prop}
 Suppose one of the singular values of the $2\times 2$ Jacobian matrix
 at the $n$-step is vanishing as in Lemma~\ref{lem:w4sv_singular},
 then the W4SV map with $\Delta\tau=1/2$
 produces an increment 
that is not aligned with $\vv_{+}$ or $\vv_{-}$ in general.
Such an alignment occurs only when
 the angle between $\vF(\vx^{(n)})$ and $\vu_{+}^{(n)}$
 accidentally satisfies a particular relation with the angle
 between $\vv_{+}^{(n)}$ and $\vv_{+}^{(n+1)}$.
 \label{prop:w4sv_singular}
\end{prop}
\begin{proof}
 The W4SV map with $\Delta\tau=1/2$ is given as
\begin{subequations}
 \label{eq:W4SVmap1/2}
 \begin{align}
  \bm{x}^{(n+1)} &= \bm{x}^{(n)} +\frac{1}{2} V \bm{p}^{(n)},\\
  \bm{p}^{(n+1)} &= -\frac{1}{2} \hat\Sigma^{-1}U^{-1}\bm{F}(\bm{x}^{(n)}).
 \end{align}
\end{subequations}
 The increment at the $(n+1)$-step can be written as
 \begin{eqnarray}
  \bm{x}^{(n+2)} -\bm{x}^{(n+1)} &=&
   -\frac{1}{4}V_{(n+1)}\hat\Sigma_{(n)}^{-1}U_{(n)}^{T}\vF(\vx^{(n)})
   = -\frac{F\cos\theta}{4\sigma_{+}^{(n)}}\vv_{+}^{(n+1)}
   -\frac{F\sin\theta}{4}\vv_{-}^{(n+1)},
 \end{eqnarray}
 where $\cos\theta$ is the angle between $\vu_{+}^{(n)}$ and
 $\vF(\vx^{(n)})$
 and $F=|\vF|$ is the absolute value of vector~$\vF$.
 With the employment of the angle~$\phi$ between $\vv_{+}^{(n)}$ and $\vv_{+}^{(n+1)}$,
 the right-singular vectors~$\vv_{+}^{(n)}$ and $\vv_{-}^{(n)}$
 at the $(n+1)$-step are expanded by those at the $n$-step as follows:
 \begin{eqnarray}
  \vv_{+}^{(n+1)} = \vv_{+}^{(n)}\cos\phi  -\vv_{-}^{(n)}\sin\phi,\nonumber\\
  \vv_{-}^{(n+1)} = \vv_{+}^{(n)}\sin\phi  +\vv_{-}^{(n)}\cos\phi.
 \end{eqnarray}
 The increment is finally expressed by the latter vectors as
 \begin{eqnarray}
  \bm{x}^{(n+2)} -\bm{x}^{(n+1)} &=&
  -\frac{F}{4\sigma_{+}^{(n)}}
   \left\{\cos\theta\cos\phi +\sigma_{+}^{(n)}\sin\theta\sin\phi
   \right\} \vv_{+}^{(n)}
   +\frac{F}{4\sigma_{+}^{(n)}}
   \left\{\cos\theta\sin\phi -\sigma_{+}^{(n)}\sin\theta\cos\phi
   \right\} \vv_{-}^{(n)}.\nonumber\\
 \end{eqnarray}
 This indicates that 
the increment in the direction of $\vv_{-}^{(n)}$ 
exists unless
$\cos\phi=\cos\theta=0$ or $\tan\phi=\sigma_{+}\tan\theta$.
To put another way, there is no increment in the direction of $\vv_{-}^{(n)}$, which corresponds to $\sigma_{-}=0$,
 only when the following condition~(i) or (ii) is satisfied:
 \begin{eqnarray}
  \begin{cases}
   \mathrm{(i)} & \vF^{(n)}\perp\vu_{+}^{(n)}\ \mathrm{and}\ \vv_{+}^{(n+1)}\perp\vv_{+}^{(n)},\\
   \mathrm{(ii)} & \phi=\tan^{-1}\left(\sigma_{+}\tan\theta\right).
  \end{cases}
  \label{eq:condition}
 \end{eqnarray}
\end{proof}
\begin{thm}
 Suppose a two dimensional problem, in which the
 associated Jacobian matrix can be defined and 
 has at least one nonzero singular value.
 Let $\vx^{(n)}$ be a series of intermediate
 solution vectors during the iteration
 defined by the W4SV map.
 Then, the W4SV map with $\Delta\tau=1/2$ can reach a solution
 from initial conditions sufficiently close to the true solution
 unless the particular condition~\eqref{eq:condition} is satisfied among
 the vectors~$\vF(\vx^{(n)}), \vu_{+}^{(n)}, \vv_{+}^{(n)}$
 and $\vv_{+}^{(n+1)}$.
\end{thm}
\begin{proof}
The claim is obtained if the
 $2\times 2$
 Jacobian matrix is non-singular from
 Proposition~\ref{prop:w4sv_nonsingular}.

 If the W4SV iteration encounters a singularity of the Jacobian
 matrix,
 the error does not decrease
 in the direction of the right-singular vector~$\vv_{-}^{n}$
 from the $n$-step to the $(n+1)$-step.
 However, by Proposition~\ref{prop:w4sv_singular},
 there appears a nonzero increment in the direction of $\vv_{-}^{(n)}$
 at the $(n+2)$-step
 unless the condition~\eqref{eq:condition}
 is satisfied.
 The series of vectors~$\vx^{(n)}$ by the W4SV iteration map 
hence converges to
 the solution even in this case.
\end{proof}

%%%%%%%%%%%%%%%%%%%%%%%%%%%%%%%%%%%%%%%%%%%%%%%%%%
\section{Numerical tests}\label{sec:results}
%%%%%%%%%%%%%%%%%%%%%%%%%%%%%%%%%%%%%%%%%%%%%%%%%%
%
\begin{table}[t]
\begin{tabular}{l|c|c|c|c|c|c|c}
 Problem No. & $\vx_0$ & NR & dNR & qN & mqN$^{(*1)}$
 & W4UL & W4LH\\\hline
 \begin{tabular}{c}
  1\
  (Rosenbrock)
 \end{tabular} &$(-1.2,1)^T$ & * & * & * & 4 & * & 45\\
 \begin{tabular}{c}
  2\
  (Freudenstein \& Roth)
 \end{tabular} & $(6,3)^T$ & 105 & 32 & 35 & 12 & 50 & $\bigtriangleup$\\
 \begin{tabular}{c}
  3\
  (Powell)
 \end{tabular} & $(0,1)^T$ & 12 & 36 & 73 & 28 & 60 & 57\\
 & $(1,1)^T$ & * & * & * & & * & * \\
 \begin{tabular}{c}
  4\
  (Brown)
 \end{tabular} & $(1,1)^T $ & 538 & 465 & * & * & 711 & $\bigtriangleup$ \\
 \begin{tabular}{c}
  5\
  (Beale)
 \end{tabular} & $(1,1)^T$ & * & * & * & * & * & 642 \\
   & $(0,2)^T$ & * & * & * &  & * & * \\
 \begin{tabular}{c}
  A\
  (Hueso \& Monteiro)
 \end{tabular} & $(1.5,2.5)^T$ & 13 & 30 & * & & 55 & 47 \\
 \begin{tabular}{c}
  B\
  (Fujisawa)
 \end{tabular} & $(0,1)^T$ & * & * & * &  & * & * \\
   & $(0,-1)^T$ & * & * & * &  & * & * 
\end{tabular}
 \caption{Numerical results obtained with representative root-finding
 methods for the two dimensional problems
 given in~\ref{sec:problem}.
The numbers in the third to sixth columns
 show how many iterations are
 needed for the different methods to obtain the solution.
 We put ``$*$'' when the method fails to find a solution
 or ``$\bigtriangleup$'' when it takes
 more than $10^6$ iterations to get the solution.
 (NR: Newton-Raphson method, dNR: damped Newton-Raphson method
 with $\Delta\tau=0.5$, qN: quasi-Newton(Good Broyden) method
 with $\Delta\tau=0.5$,
 mqN$^{(*1)}$: modified quasi-Newton method by Fang \textit{et al.}\cite{fang2018modified},
 W4UL: W4 method with the UL decomposition by Okawa \textit{et al.}\cite{Okawa:2018smx},
 W4LH: W4 method with the LH decomposition by Fujisawa \textit{et al.}\cite{Fujisawa:2018dnh}). }
 \label{table:tests_old}
\end{table}
A large number of optimization problems were
collected so far
to test the reliability and robustness of a new scheme, e.g.,~\cite{more1981}.
We first show in Table~\ref{table:tests_old} 
the performance of some representative root-finding methods
for the standard problems given
in the literature(For example, see~\cite{fang2018modified}).
 We also included two other singular problems to the list: problem A by
 Hueso\&Monteiro\cite{hueso2009} and problem B by Okawa \textit{et al.}\cite{Okawa:2018smx}.
 Each row in the table corresponds to one of these problems.
 The leftmost column gives the problem numbers,
 which are referred to also
 in~\ref{sec:problem}
 and the second column denotes the initial conditions employed,
 results are given from the third to sixth columns
 for the different methods: (from left)
 Newton-Raphson(NR) method, damped Newton-Raphson(dNR) method with $\Delta\tau=0.5$,
 Good Broyden method,which belongs to the class of quasi-Newton(qN) methods,
 modified quasi-Newton(mqN) method by Fang \textit{et al.}\cite{fang2018modified},
 W4 method with the UL decomposition and $\Delta\tau=0.5$ by Okawa \textit{et al.}\cite{Okawa:2018smx},
 and W4 method with the LH decomposition and $\Delta\tau=0.5$ by Fujisawa \textit{et
 al.}\cite{Fujisawa:2018dnh}, respectively.
 We display in each cell of these columns the number of iteration steps it takes the corresponding method
 to reach the solution within the error of $10^{-8}$,
 which is defined as $\mathrm{max}(F_x/||F_x||,F_y/||F_y||)$ with
 $||F_i||$ being the sum of absolute values of all terms in $F_i$.
 We put instead ``*'' when the method fails to find a solution
 entirely or ``$\bigtriangleup$'' when 
the solution is obtained with more than $10^6$ iterations.

It should be clear that these problems are 
actually very difficult to solve.
In fact, none of the methods in the list
is able to obtain the solution for all the problems.
Note that some of the problems have
a singular Jacobian for the initial condition
and the iteration map cannot be defined
in the first place.
The details will be given below for this type of problems.

\begin{table}[t]
\begin{tabular}{l|c|c|c|c|c|c}
 Problem No. & $\vx_0$ 
 & \begin{tabular}{c}W4SV\\ ($\Delta\tau=1$) \end{tabular}
 & \begin{tabular}{c}W4SV\\ ($\Delta\tau=0.9$) \end{tabular}
 & \begin{tabular}{c}W4SV\\ ($\Delta\tau=0.8$) \end{tabular}
 & \begin{tabular}{c}W4SV\\ ($\Delta\tau=0.7$) \end{tabular}
 & \begin{tabular}{c}W4SV\\ ($\Delta\tau=0.5$) \end{tabular}\\\hline
 \begin{tabular}{l}
  1\
  % (Rosenbrock)
 \end{tabular}
 &$(-1.2,1)^T$ & 4 & 19 & 31 & 30 & 40 \\
 \begin{tabular}{l}
  2\
  % (Freudenstein \& Roth)
 \end{tabular}
 & $(6,3)^T$ & 210 & 95 & 72 & 58 & 50 \\
 \begin{tabular}{l}
  3\
  % (Powell)
 \end{tabular}
  & $(0,1)^T$ & 24 & 29 & 34 & 40 & 58 \\
  & $(1,1)^T$ & 42 & 155 & 61 & 75 & 154 \\
 \begin{tabular}{l}
  4\
  % (Brown)
 \end{tabular}
  & $(1,1)^T $ & 188 & 33136 & 3279 & 3621 & 8266 \\
 \begin{tabular}{l}
  5\
  % (Beale)
 \end{tabular}
  & $(1,1)^T$ & 12 & 15 & 18 & 22 & 37\\
   & $(0,2)^T$ & 16 & 30 & 381 & 34 & 58\\
 \begin{tabular}{l}
  A\
  % (Hueso \& Monteiro)
 \end{tabular}
  & $(1.5,2.5)^T$ & 26 & 29 & 33 & 38 & 55\\
 \begin{tabular}{l}
  B\
  % (Fujisawa)
 \end{tabular}
 & $(0,1)^T$ & 10 & 14 & 18 & 14 & 43\\
 & $(0,-1)^T$ & * & 56 & 28 & 38 & 307
\end{tabular}
 \caption{
 Same as Table~\ref{table:tests_old} but for our new scheme
with three different values of $\Delta\tau$.
 }
 \label{table:tests_new}
\end{table}
Before proceeding to the details,
we exhibit the corresponding results obtained with our new W4SV
method in Table~\ref{table:tests_new}.
Note that there is one free parameter, the time interval~$\Delta\tau$ in the W4 schemes.
In the table we show the results for three cases with
$\Delta\tau=1, 0.9, 0.8, 0.7$ and $0.5$\footnote{Although our analysis in this paper is valid only in the range of $0<\Delta\tau<1$, we have just included $\Delta\tau=1$ to directly compare the W4SV method with the NR method.  It indicates that it is safer to employ $\Delta\tau<1$.}.
It is remarkable that the W4SV method can solve
essentially all the problems including those with the singular
Jacobians in the initial conditions.
This clearly indicates that the W4SV method is good at finding roots
for all sorts of two dimensional nonlinear problems\footnote{As demonstrations, the W4SV method for these problems are also written with Python language and are now public in \cite{OkawaPyW4}.}.
Now we move on to the details of each singular problem.% in the following.

\subsection{Powell's Badly Scaled function}
\begin{figure}[t]
 \begin{tabular}{cc}
  \includegraphics[width=8cm,clip]{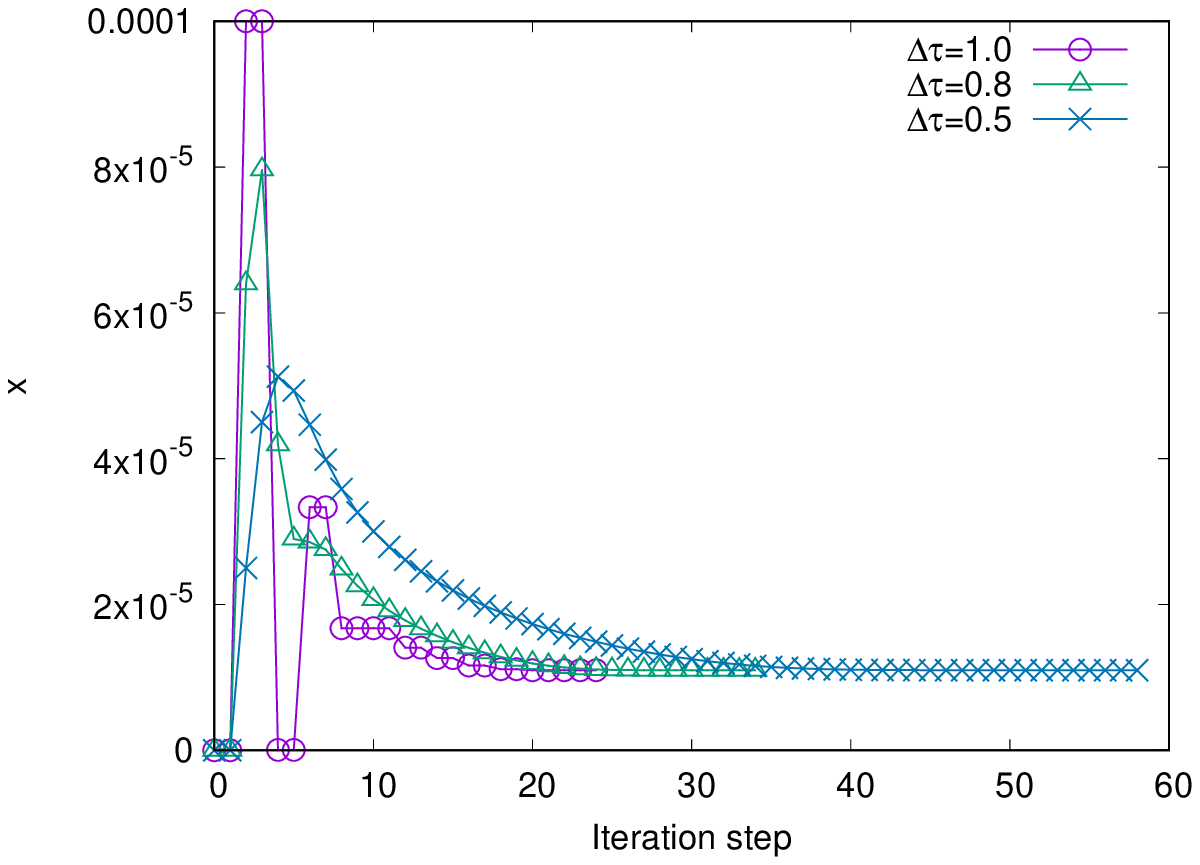} &
  \includegraphics[width=8cm,clip]{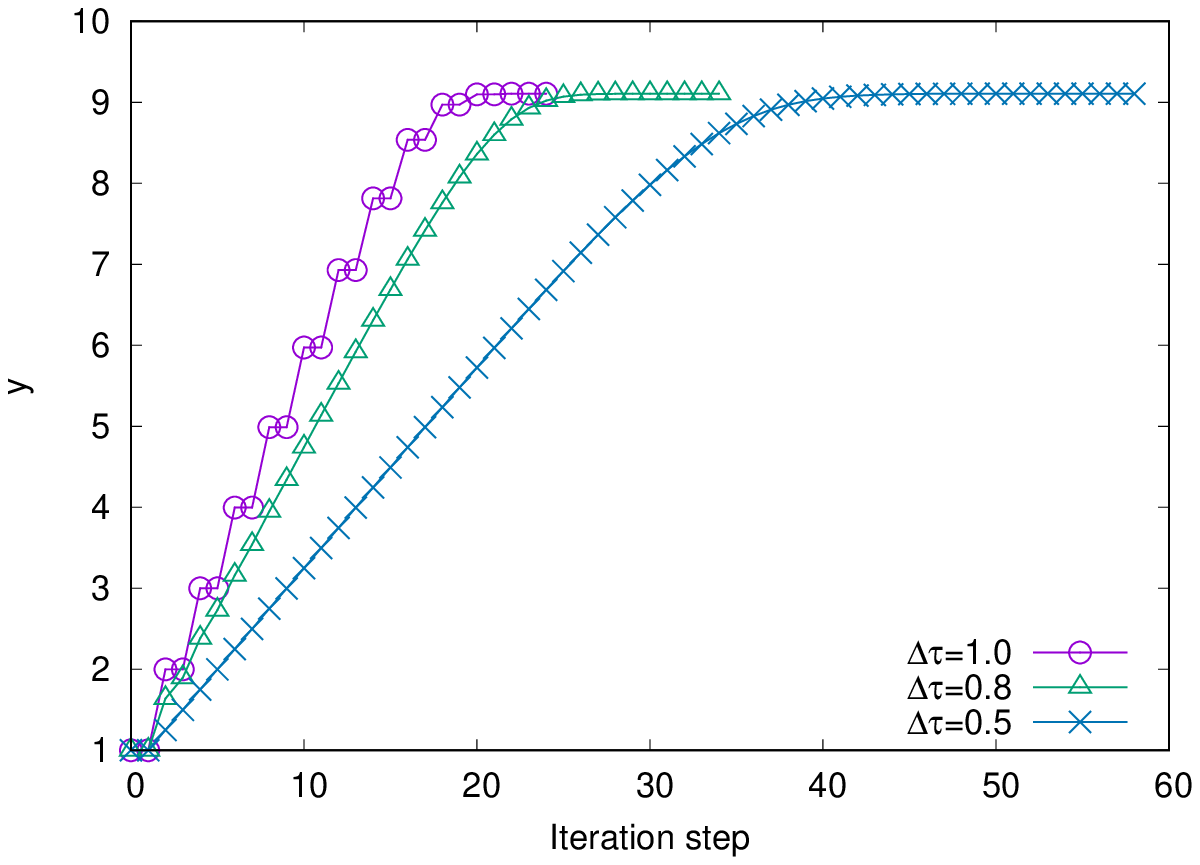}\\
  (a) & (b)\\
  \includegraphics[width=8cm,clip]{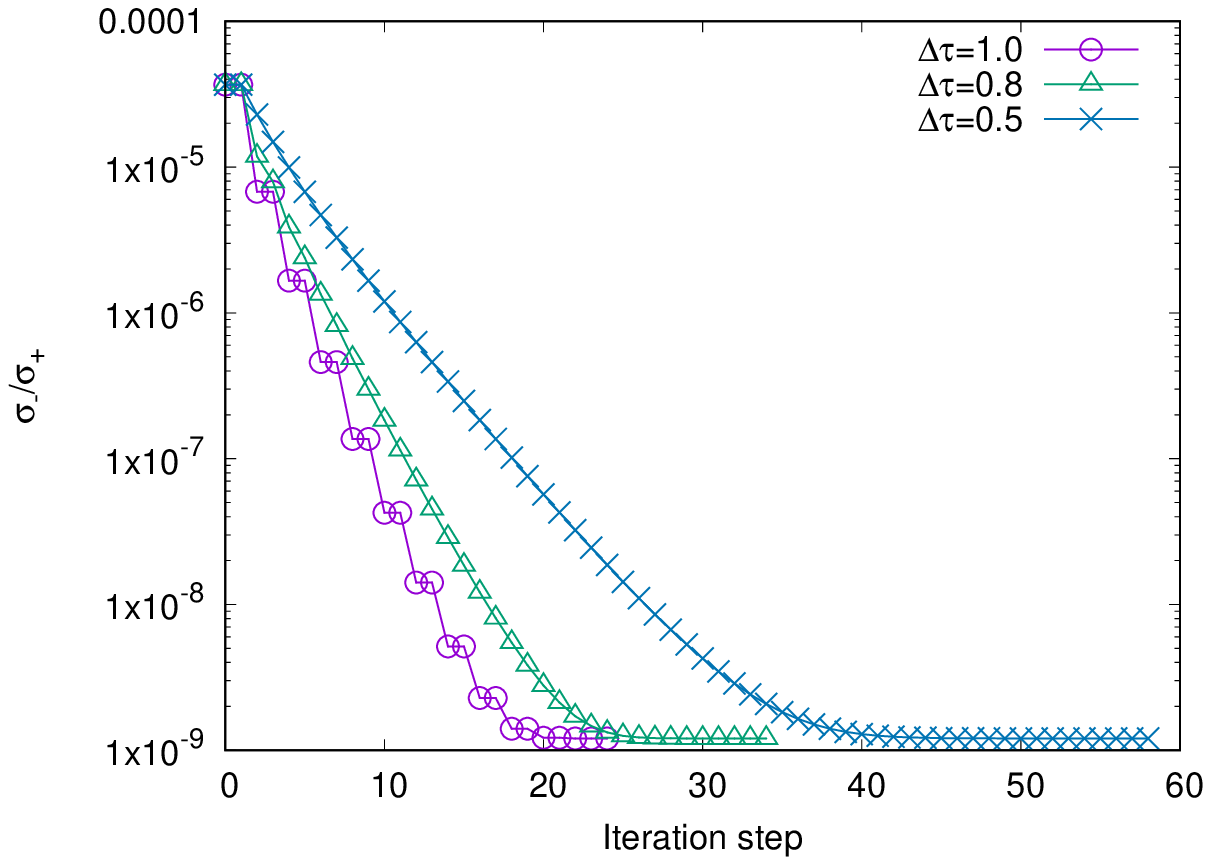} &
  \includegraphics[width=8cm,clip]{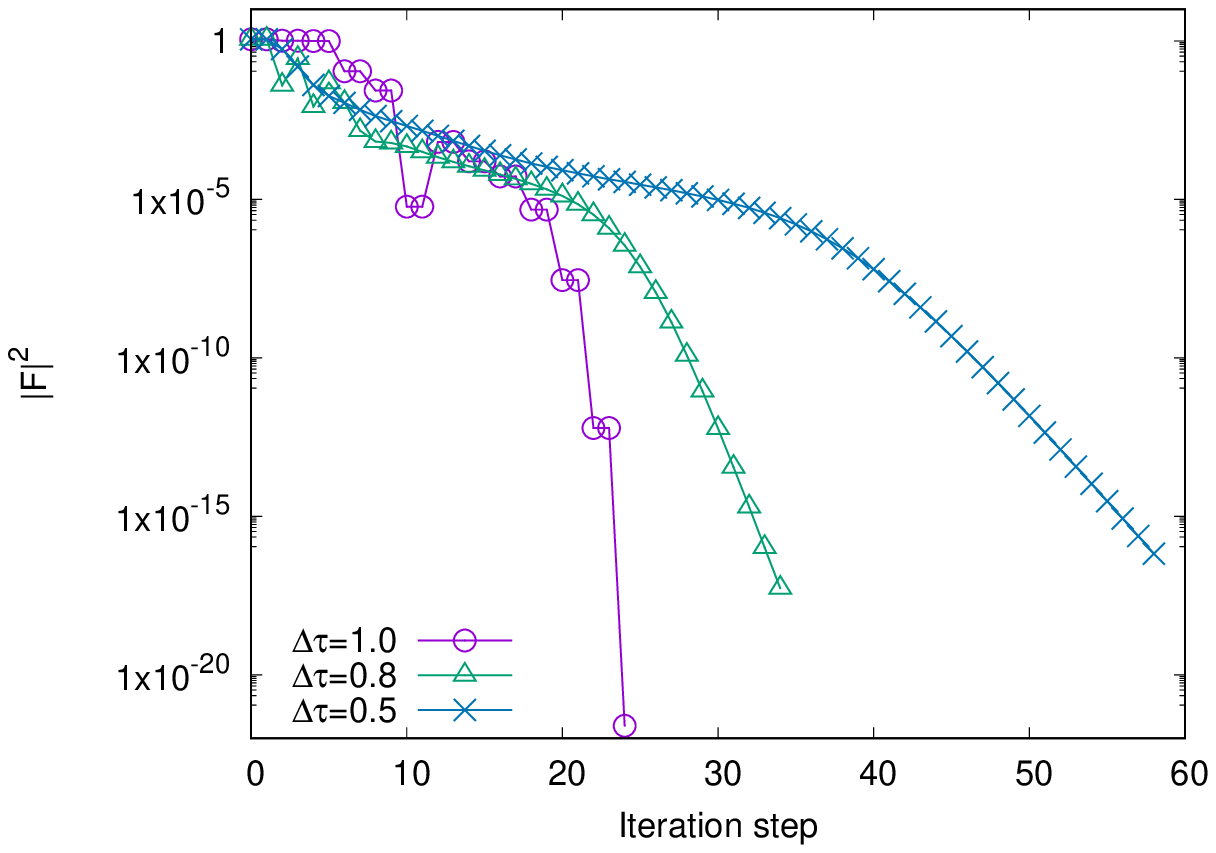}\\
  (c) & (d)
 \end{tabular}
 \caption{The numerical results obtained with the
 W4SV method
 for Powell's badly scaled problem: %with initial condition~$(0,1)^T$
 (a) $x$, (b) $y$,
 (c) the ratio of the smaller to larger singular
 values~$\sigma_{-}/\sigma_{+}$,
 (d) the error norm~$|F|^2 = f_x^2 +f_y^2$.
 The different colors specify the values of $\Delta\tau$.}
 \label{fig:powell1}
\end{figure}
The singularity in the initial condition cannot be
treated with the existing methods that need
to invert the Jacobian matrix.
We show in the following subsections how such
singularities are handled with the W4SV method.
The first example is
the well-known
 Powell's Badly Scaled function~\eqref{eq:powell}.
If we choose the initial condition as~$\vx=(0,1)^T$,
then even the NR method obtains the solution rather easily.
We are faced with a
difficulty, however, when we start the NR iteration 
from $\vx=(1,1)^T$ as in Table~\ref{table:tests_old}.
In fact, the Jacobian matrix associated with this problem is
given at $\bx=(x,y)^T$ as
\begin{eqnarray}
 J=
 \begin{pmatrix}
  10^4y  & 10^4x \\
  -e^{-x} & -e^{-y}
 \end{pmatrix}.
\end{eqnarray}
Obviously, it is singular when $x=y$.
In the existing root-finding methods using the inversion of Jacobian,
one has to change the initial condition.
As we show shortly, our W4SV method can solve the problem even from
this initial condition.

Fig.~\ref{fig:powell1} shows the evolutions in the W4SV method of (a) $x$, (b) $y$,
(c) $\sigma_{-}/\sigma_{+}$ the ratio of singular values
 and (d) $|F|^2 = f_x^2 +f_y^2$ the norm of error
 for Powell's badly scaled function.
The initial condition is $\vx=(0,1)^T$,
which is non-singular.  The different colors in the figure
correspond to the choice from $\Delta\tau=1, 0.8, 0.5$.
It is found from panels
 (a) and (b) of Fig.~\ref{fig:powell1}
 that $x$ and $y$ are settled down to the solution for all the values of $\Delta\tau$.
 As shown in panel (c),
the smaller singular value~$\sigma_{-}$ 
approaches a small value as the iteration is closing to an end,
while the iteration map is well-defined there.
 As expected from the analysis of the W4SV map
near the solution in Lemma~\ref{lem:error},
 the error decreases monotonically in Fig.~\ref{fig:powell1}
 after the numerical solution $\bx$ comes close to
the solution.

\begin{figure}[t]
 \begin{tabular}{cc}
  \includegraphics[width=8.cm,clip]{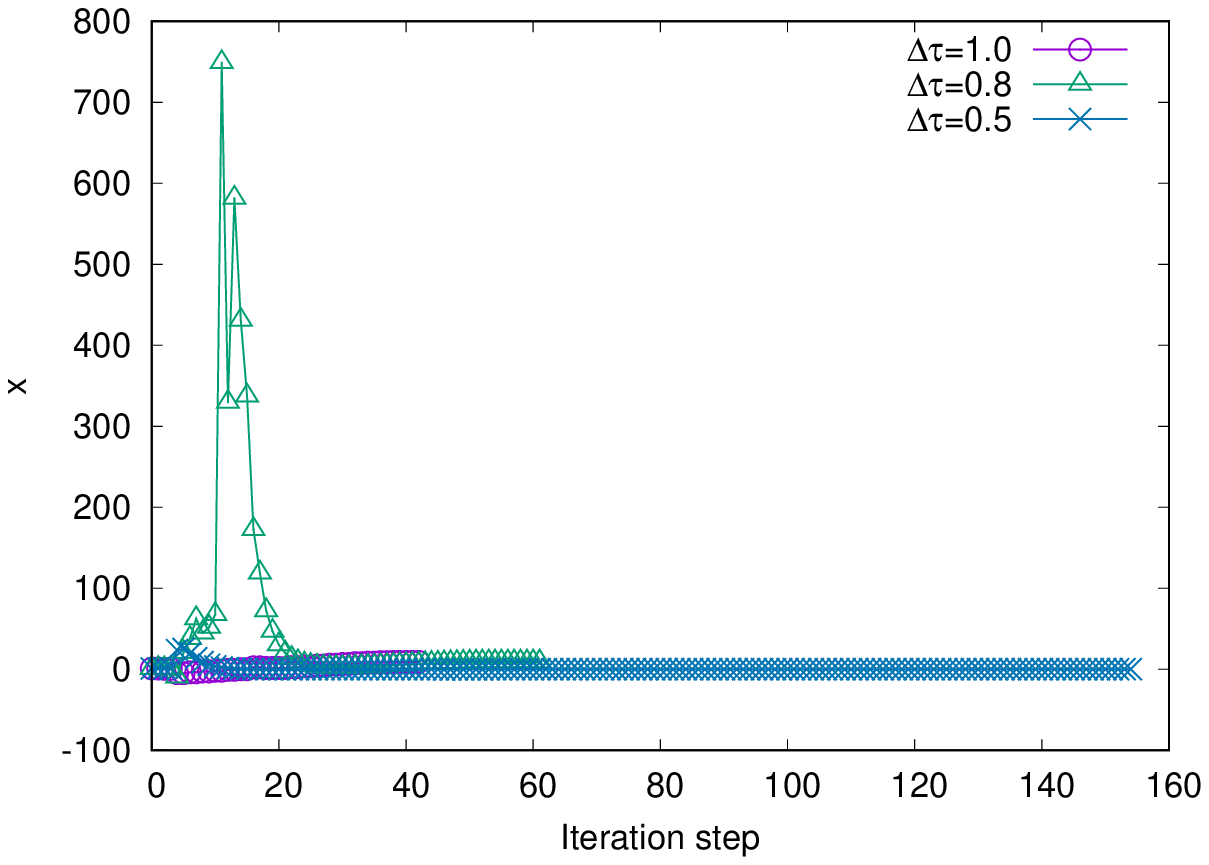} &
  \includegraphics[width=8.cm,clip]{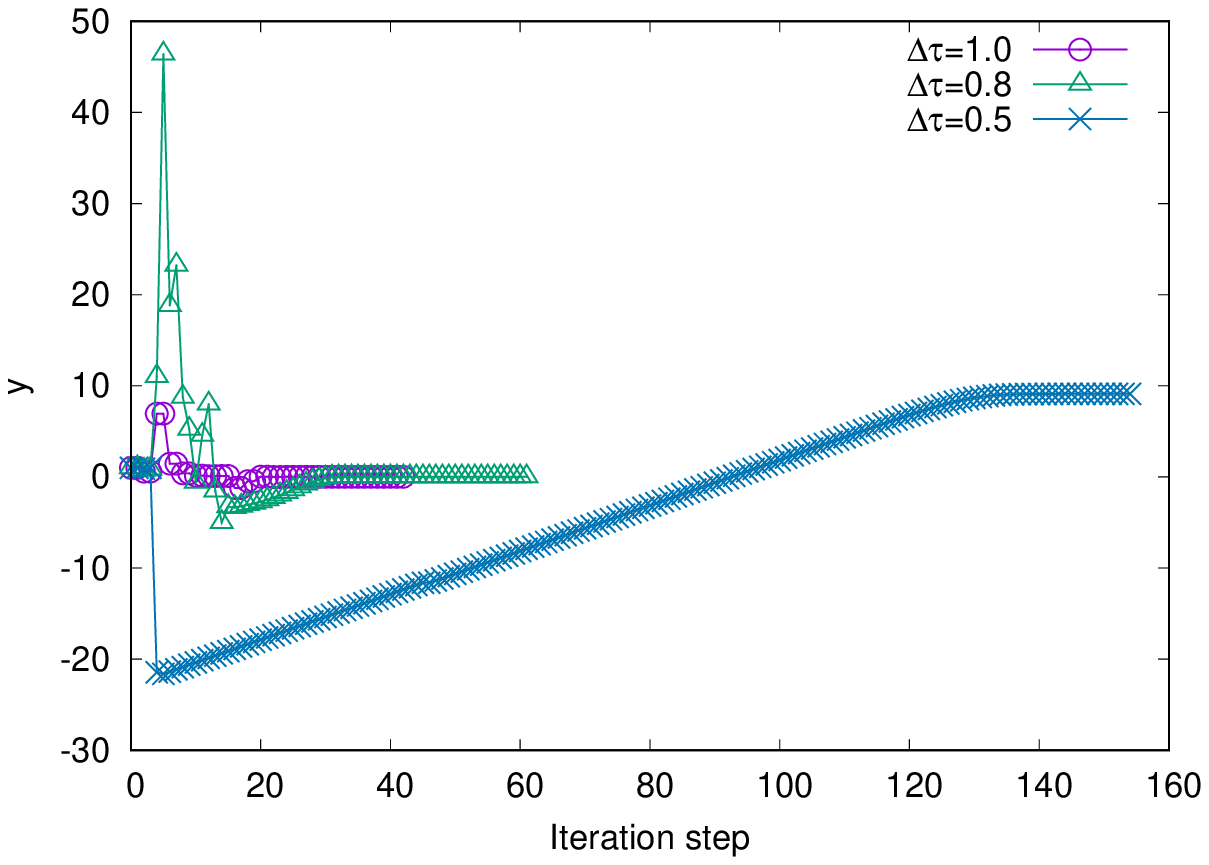}\\
  (a) & (b)\\
  \includegraphics[width=8.cm,clip]{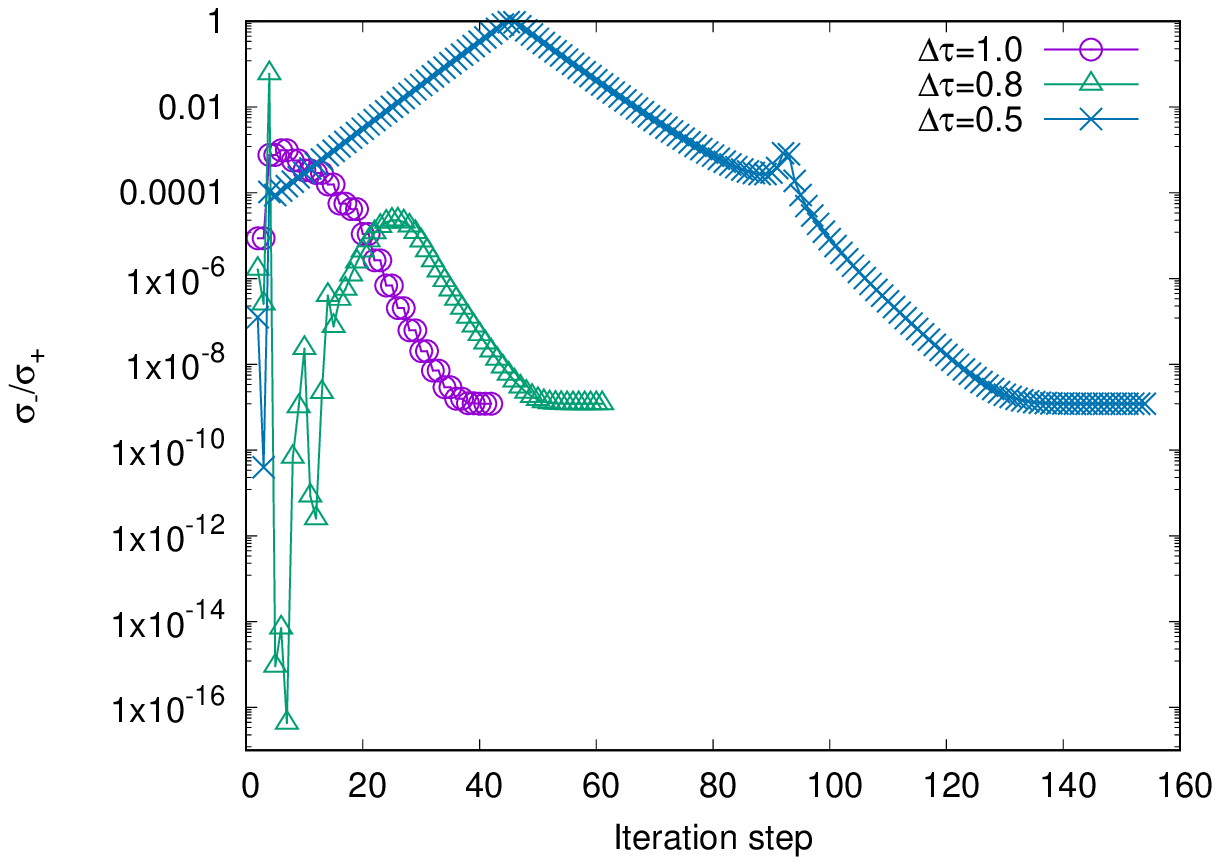} &
  \includegraphics[width=8.cm,clip]{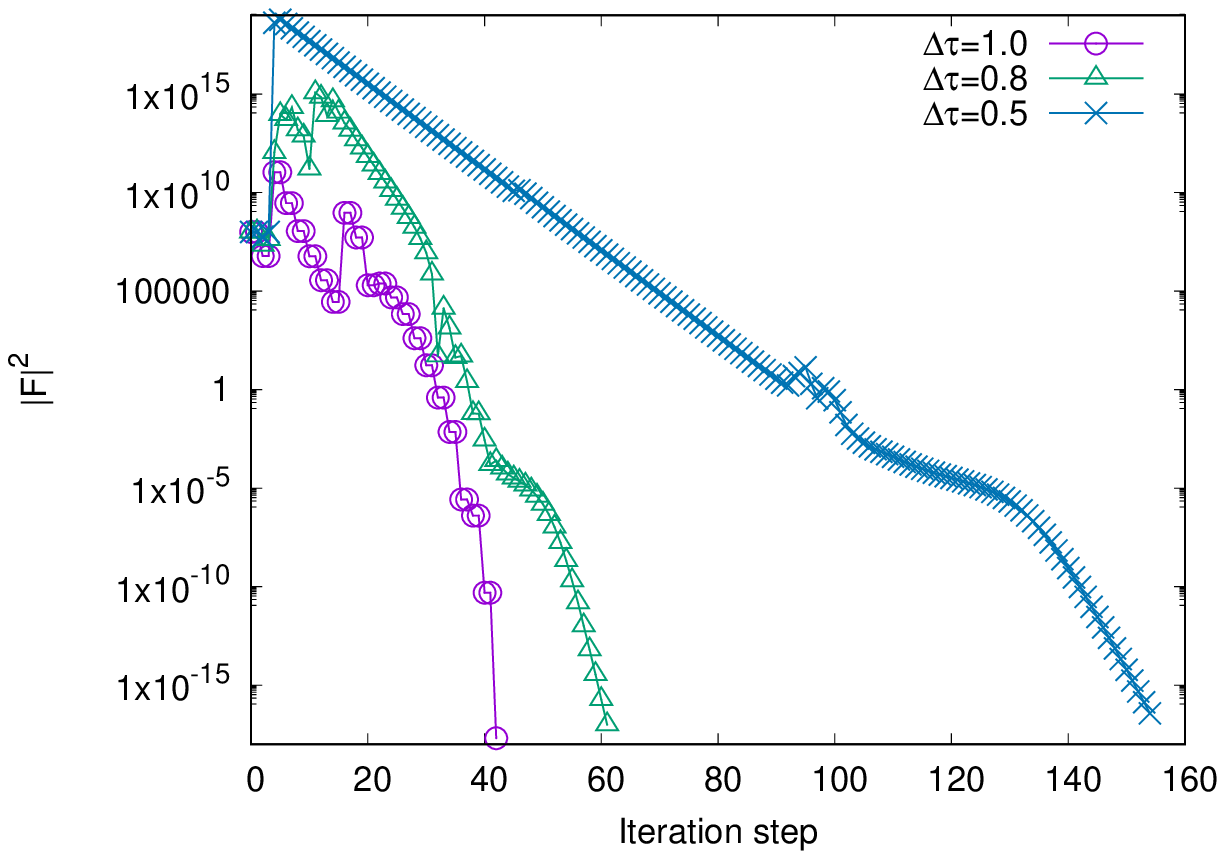}\\
  (c) & (d)
 \end{tabular}
 \caption{Same as Fig.~\ref{fig:powell1}
 except for the initial condition~$(1,1)^T$.}
 \label{fig:powell2}
\end{figure}
In Fig.~\ref{fig:powell2}, we display the same quantities
as in Fig.~\ref{fig:powell1} except for the initial
guess, which is now set to $(1,1)^T$.
This initial condition is singular
and the application of the existing methods with the Jacobian
inversion is simply impossible.
Remarkably, the W4SV method successfully found the solution
also in this case for all three values of $\Delta\tau$.
As expected,
the error decreases monotonically towards the end of iterations,
as seen in (d) of Fig.~\ref{fig:powell2}.
More importantly, the initial condition is singular
~(see panel (c)) and is much farther away from the solution
~(panel (d)) compared with the above case.
Note that again for all three values of $\Delta\tau$
the W4SV method
defines non-singular iteration map
 and successfully escapes from the initial singularity.
It is observed, however, that the iteration maps
kick the intermediate solutions in early iteration steps
away from the solution.
This happens because 
the intermediate solutions
stay
 near the initial singilarity after a few iterations 
 and the big factor~$1/\sigma_{-}$ combined
 with large values of $|F|$
 in the iteration map finally pushes
 the following solution away from the true solution.
Nevertheless, once escaped 
from the vicinity of the initial singularity,
the subsequent solutions start to move in 
the right direction and eventually converge to the solution.
Note also that this system~\eqref{eq:powell} is symmetric with respect to $x$ and $y$
and then there is another solution by exchanging $y$ for $x$ and $x$ for $y$ as results with $\Delta\tau=0.8$ and $\Delta\tau=1.0$ in Fig.~\ref{fig:powell2}.
\subsection{Beale's function}
Beale's function is given in Eq.~\eqref{eq:beale} and
the associated Jacobian matrix is 
\begin{eqnarray}
 J=
 \begin{pmatrix}
  y-1 & x \\
  y^2-1 & 2xy
 \end{pmatrix}.
\end{eqnarray}
Since $\det J = x(y-1)^2$,
the Jacobian matrix is singular
at $x=0$ or $y=1$,
corresponding to the column, respectively.
If such a singular state is encountered during the iteration,
the existing methods that utilizes the inverse of $J$
will be stuck there.
As demonstrated in Fig.~\ref{fig:beale},
which shows
the same quantities as Fig.~\ref{fig:powell1}
obtained with the W4SV method with $\Delta\tau=0.5$ for this problem,
our new method is able to reach the right answer without taking any special measure.
We choose two singular states as the initial condition:
$(1,1)^T$ and $(0,2)^T$, the results of
which are displayed with circles and pluses, respectively.
\begin{figure}[t]
 \begin{tabular}{cc}
  \includegraphics[width=8.cm,clip]{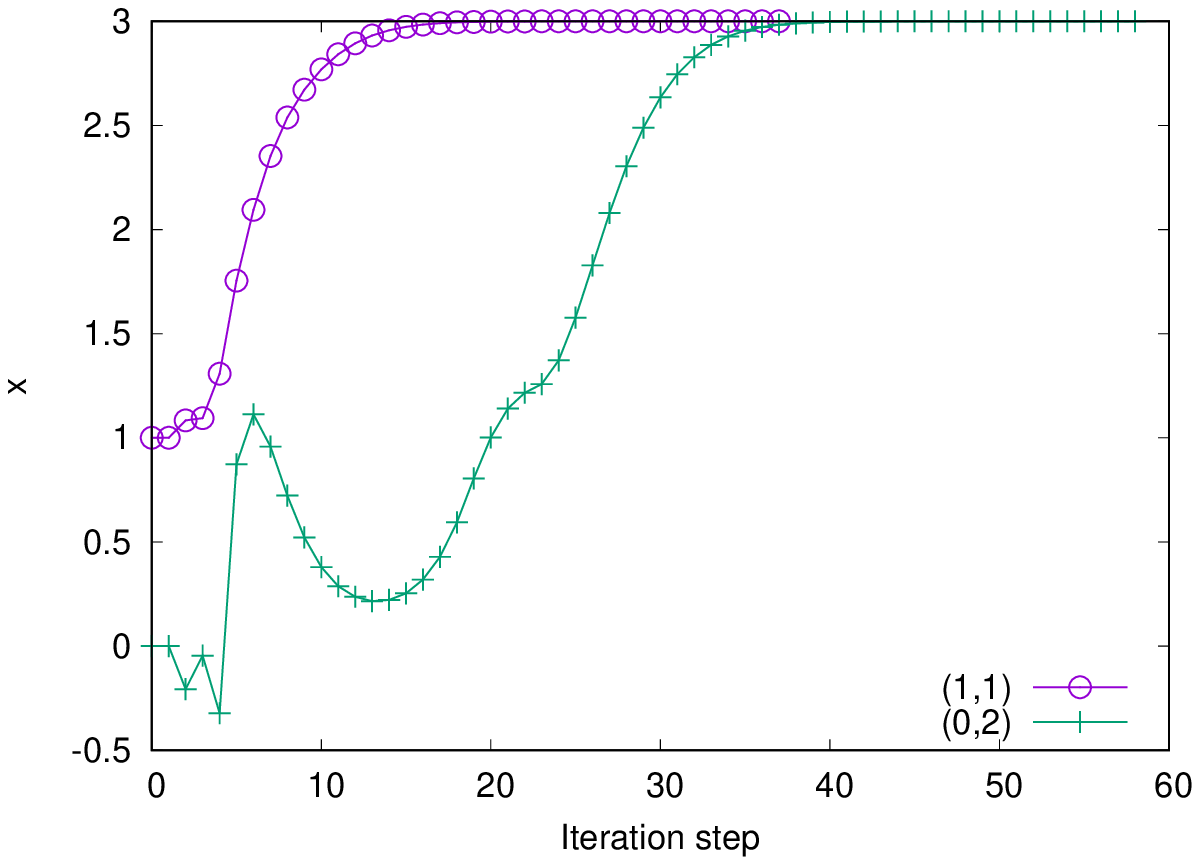} &
  \includegraphics[width=8.cm,clip]{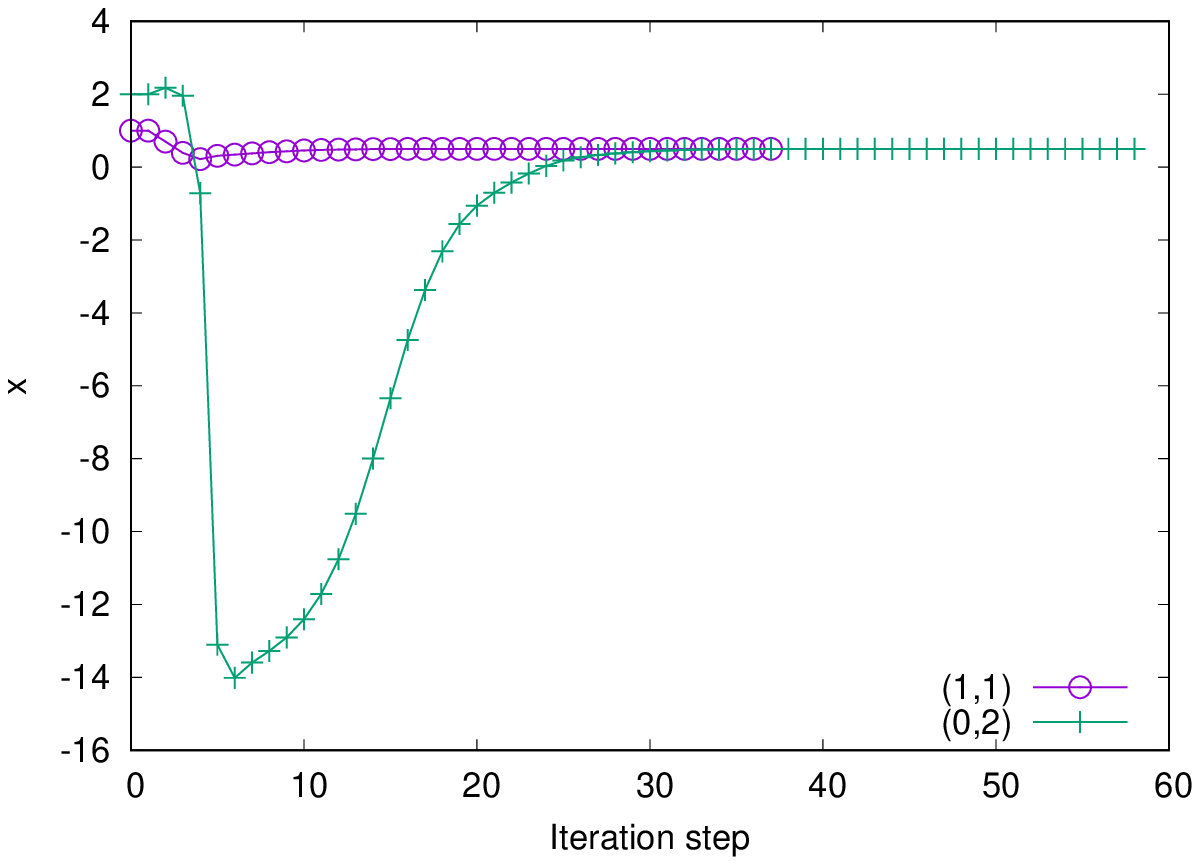} \\
  (a) & (b)\\
  \includegraphics[width=8.cm,clip]{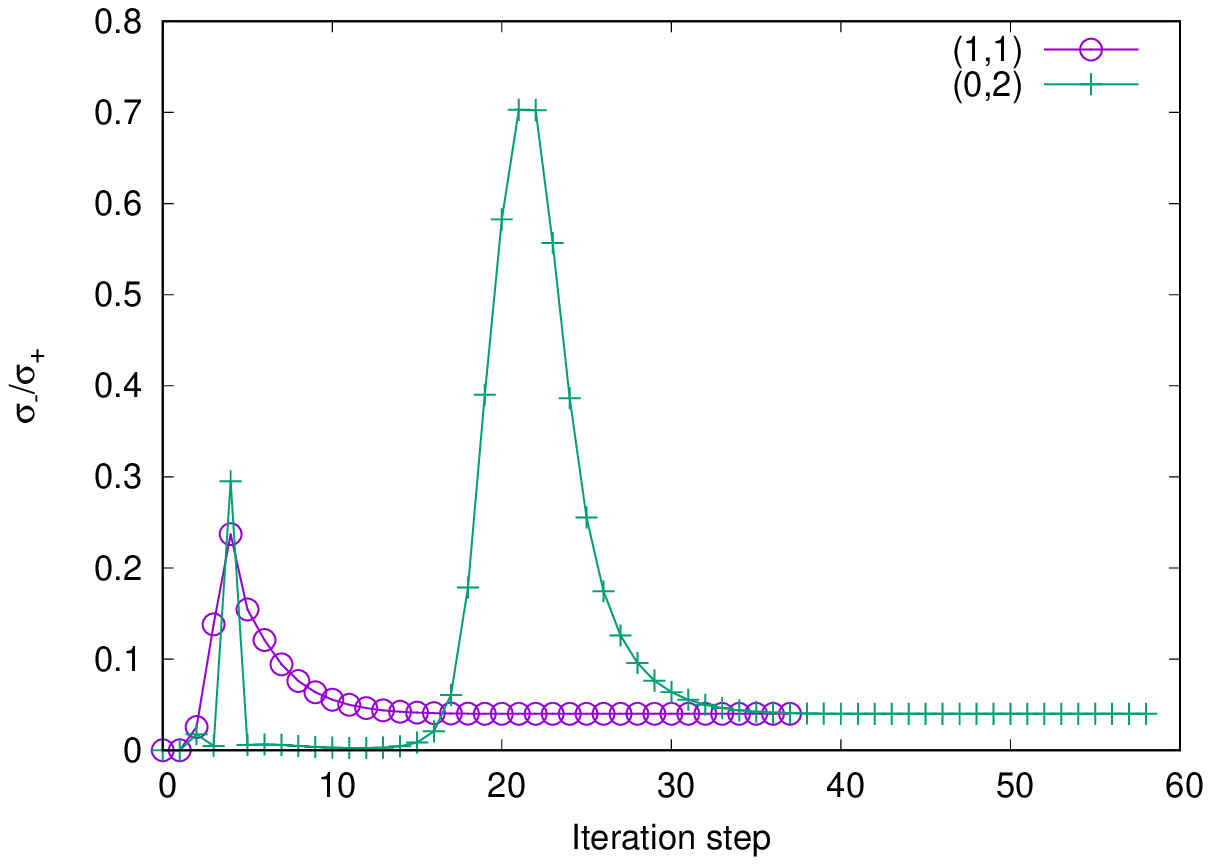} &
  \includegraphics[width=8.cm,clip]{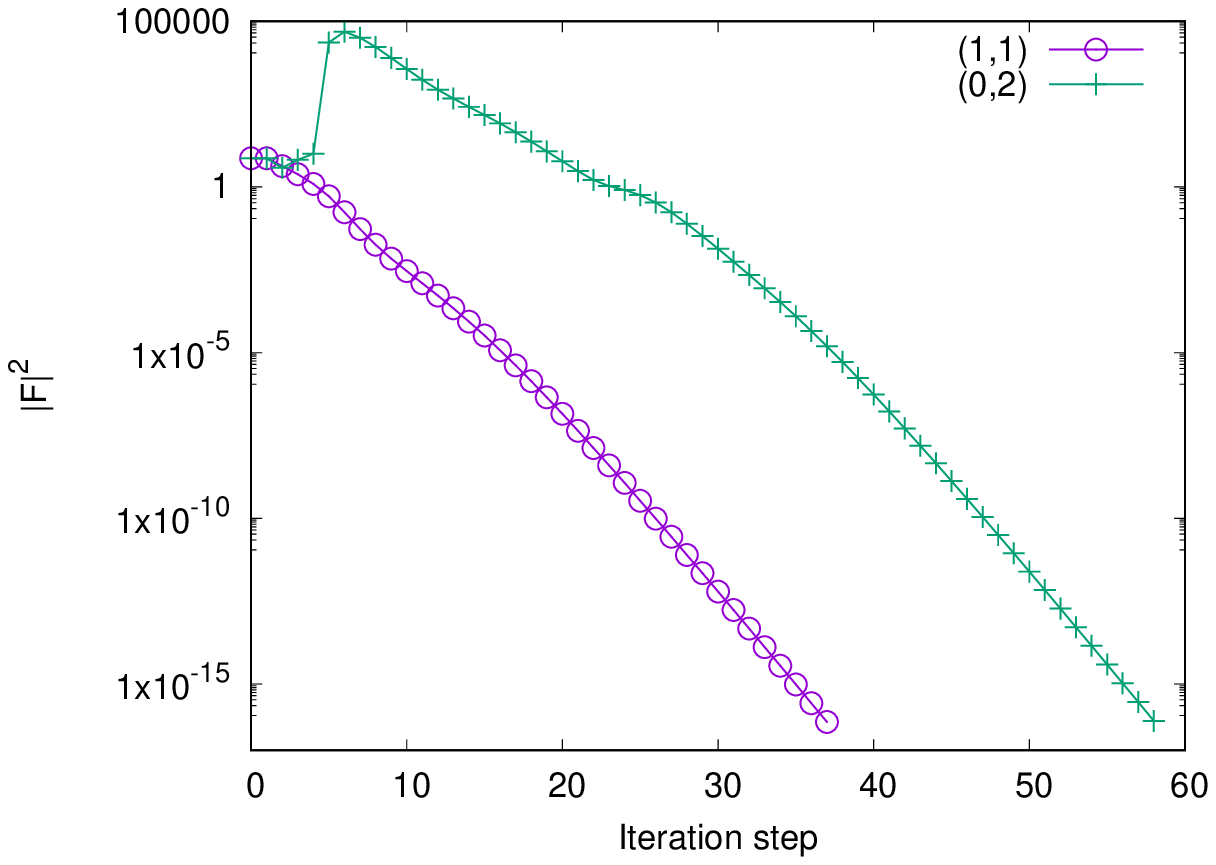} \\
  (c) & (d)
 \end{tabular}
 \caption{Same as Figs.~\ref{fig:powell1} and \ref{fig:powell2}
 but for Beale's problem~\eqref{eq:beale}.
 We choose $\Delta\tau=0.5$ for this problem.
 Circles and pluses denote the results 
 for the different initial conditions:
 $(x,y)^T=(1,1)^T$ and $(x,y)^T=(0,2)^T$, respectively.
 }
 \label{fig:beale}
\end{figure}

\subsection{Fujisawa's function}
In order to further check the capability of singular avoidance
by the W4SV method,
we apply it to another numerically tough problem 
in Eq.~\eqref{eq:fujisawa}, which we considered in our previous
paper~\cite{Okawa:2018smx}.
The Jacobian matrix is expressed as
\begin{eqnarray}
 J=
 \begin{pmatrix}
  2x & 2y \\
  2xy & x^2
 \end{pmatrix},
\end{eqnarray}
the determinant of which is $\det J = 2x(x^2-2y^2)$.
The initial condition with either $x=0$ or $x=\pm\sqrt{2}y$
is hence singular.
Even if the initial condition is nonsingular,
we know that this problem is difficult to solve numerically
 from initial conditions given below the lines~$x=\pm\sqrt{2}y$,
because it is likely that one of these singular points is encountered at
some intermediate step during the iteration~\cite{Okawa:2018smx}.
In Fig.~\ref{fig:fujisawa}, we show the results obtained with the W4SV method with $\Delta\tau=0.5$.
The initial condition is either~$(0,1)^T$, or $(0,-1)^T$.
It is obvious that small singular values,
that are encountered from time to time during the iterations,
affect the intermediate solution at the next iteration step
particularly
when they are still far from the true solution as seen, for
example,
 at around the $30\mathrm{th}$ iteration step in Fig.~\ref{fig:fujisawa}.
It is also found, however, the W4SV method can avoid such situations
and reach the right solution eventually.
As in the previous case,
the error decreases monotonically in the vicinity of the solution.
\begin{figure}[t]
 \begin{tabular}{cc}
  \includegraphics[width=8.cm,clip]{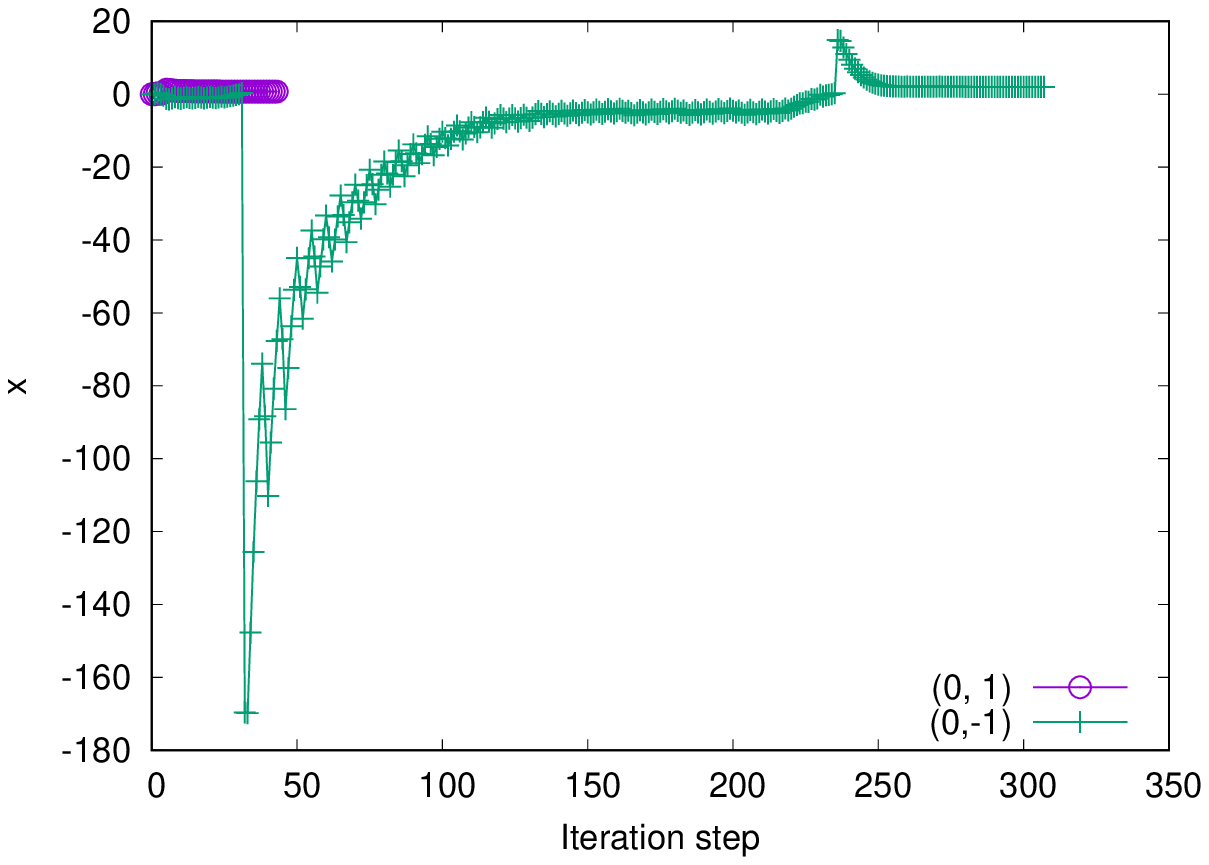} &
  \includegraphics[width=8.cm,clip]{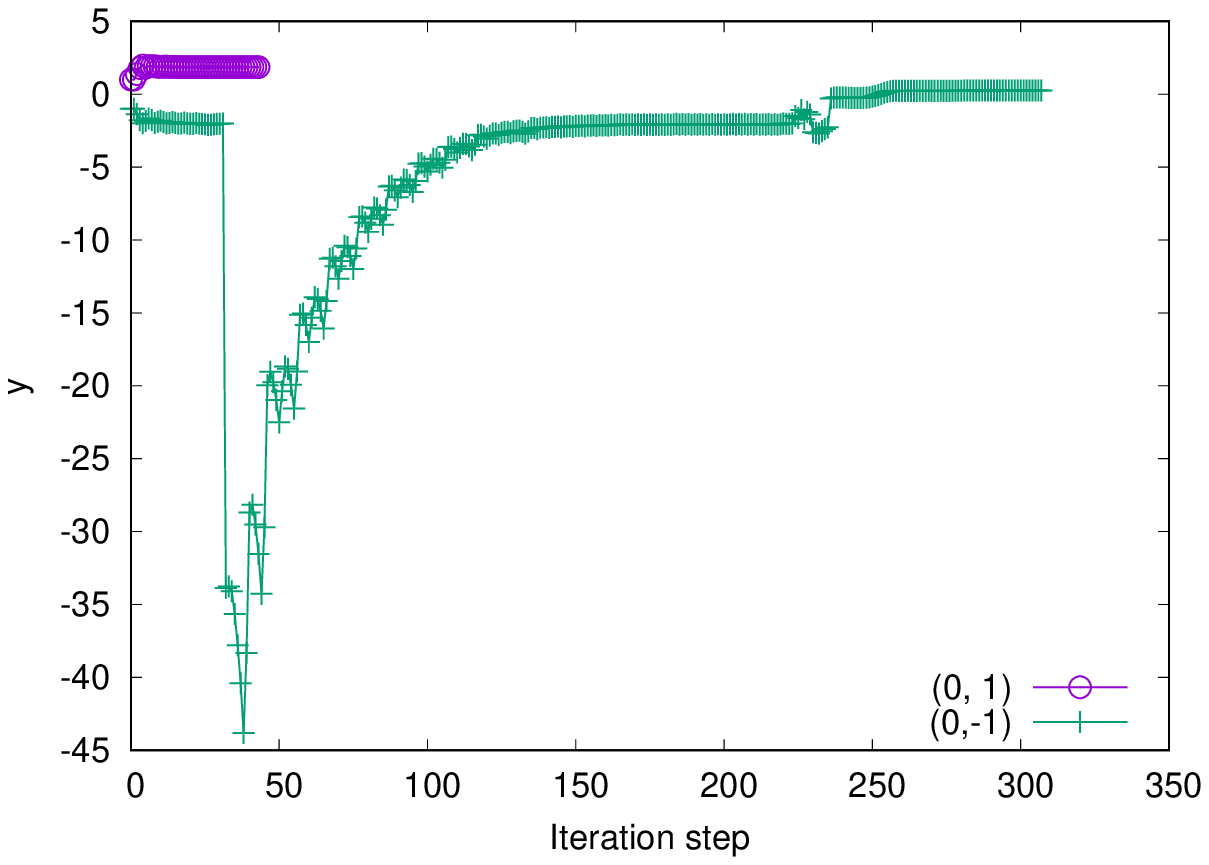}\\
  (a) & (b)\\
  \includegraphics[width=8.cm,clip]{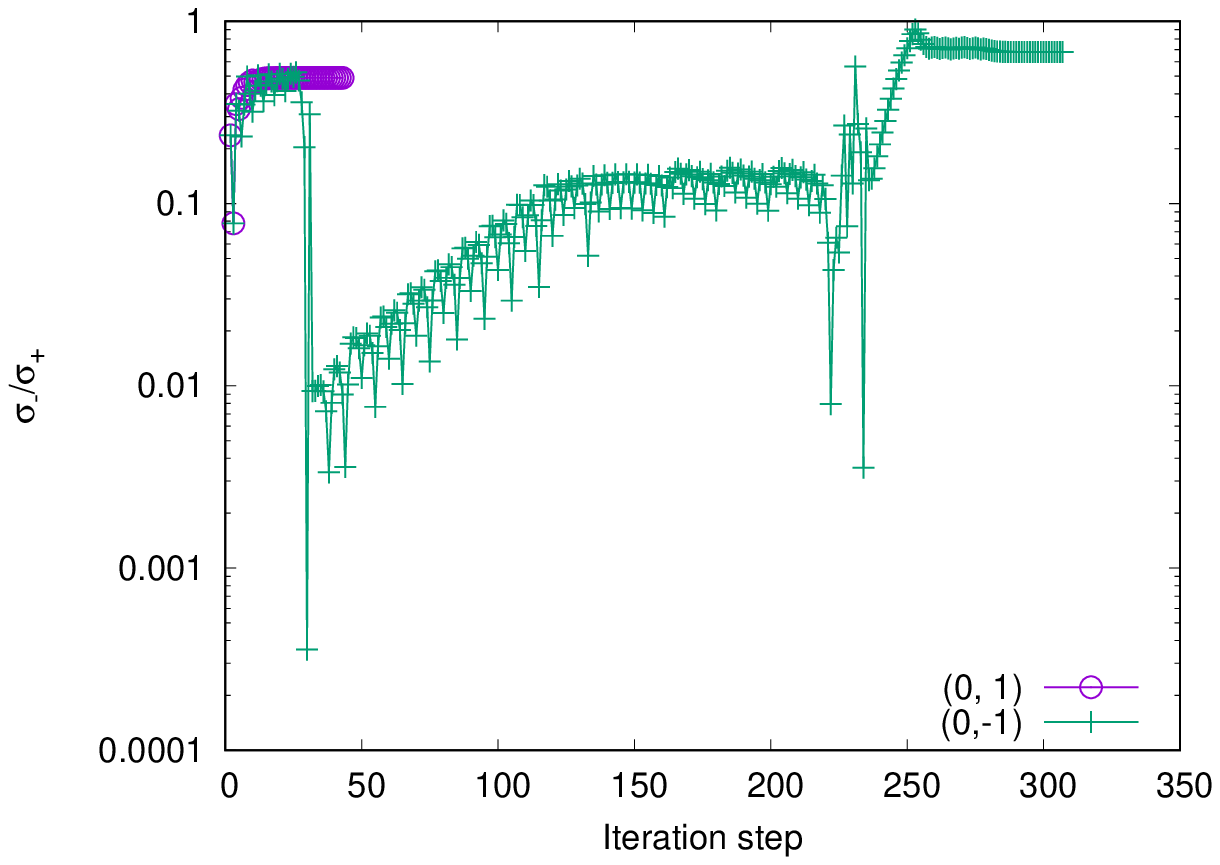} &
  \includegraphics[width=8.cm,clip]{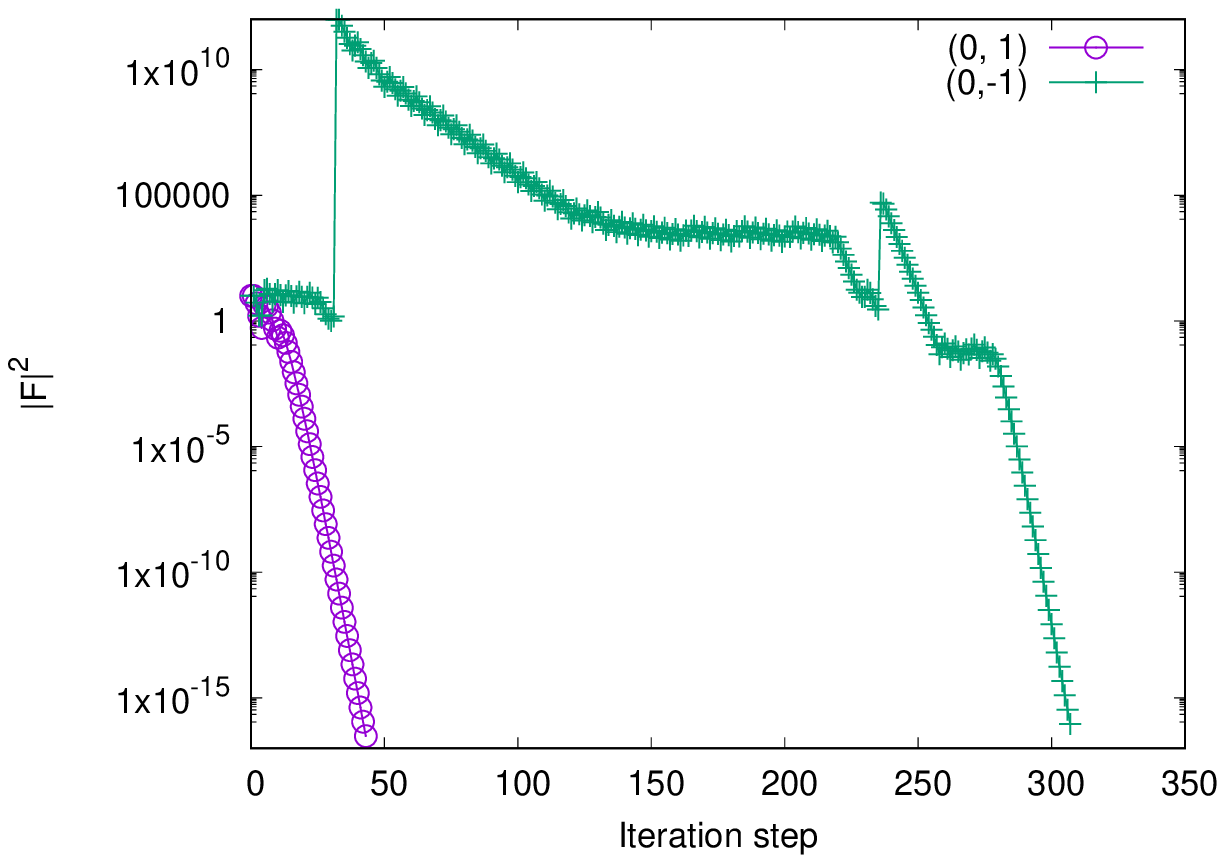}\\
  (c) & (d)
 \end{tabular}
 \caption{Same as Figs.~\ref{fig:powell1}-\ref{fig:beale}
 but for Fujisawa's function. We set $\Delta\tau=0.5$ for this test.
 The initial condition is either ~$(0,1)^T$ or $(0,-1)^T$.
 }
 \label{fig:fujisawa}
\end{figure}
%

%%%%%%%%%%%%%%%%%%%%%%%%%%%%%%%%%%%%%%%%%%%%%%%%%
\section{Conclusion}\label{sec:conclusion}
%%%%%%%%%%%%%%%%%%%%%%%%%%%%%%%%%%%%%%%%%%%%%%%%%
%
In this article, we have proposed a new scheme to solve
a set of nonlinear equations.
It is an extension of the W4 method,
 a root-finder of our own devising that
shows a nice global convergence and obtain solutions
to various problems,
for which existing methods such as the Newton-Raphson method
failed~\cite{Okawa:2018smx}.
The original W4 method has been successfully
applied to
different physics problems 
so far~\cite{Fujisawa:2018dnh,Suzuki:2020zbg,Hirai:2020sjg}.
The extension reported in this article
is meant to deal with the singular Jacobian,
which we frequently encounter in practical applications
and even the original W4 method finds difficulties
in treating.
This problem is actually common to all the existing
root-finders that require the inversion of Jacobian.
In the W4SV method, however,
the iteration map is always well-defined and can reach the
solution even for problems, for which
the Jacobian is singular at the initial or intermediate step
of iterations or at the true solution
unless the very special condition~\eqref{eq:condition}
is satisfied.
The results of the numerical tests in Sec.~\ref{sec:results}
for the well-known problems, albeit in two dimensions,
strongly support the excellent capability of our new scheme.

In principle, our new scheme should be applicable to
larger-dimensional problems, which often appear in computational science, physics and engineering.
For efficient applications to those problems, however,
it is necessary to
(i) treat non-singular but ill-conditioned Jacobians 
more efficiently
and (ii) reduce the computational cost in calculating
singular vectors and singular values,
which is the most cost-consuming part when the number of variables becomes larger.
We will address these issues in the near future.

%%%%%%%%%%%%%%%%%%%%%%%%%%%%%%%%%%%%%%%%%%%%%%%%%%%%%%%%%%%%%%%%%%%%
\section*{Acknowledgements}
%%%%%%%%%%%%%%%%%%%%%%%%%%%%%%%%%%%%%%%%%%%%%%%%%%%%%%%%%%%%%%%%%%%%
The work was supported by JSPS KAKENHI Grant Numbers JP20K14512,
JP20K03953, JP20H04728
and by Waseda University Grant for Special Research Projects(Project
number: 2019C-640 and 2020-C273).
S.Y. is supported by Institute for Advanced Theoretical and Experimental
Physics.

\appendix
\section{Test problems}\label{sec:problem}
In the literature, there are many test problems
for root-finders of nonlinear equation systems.
We summarize here
the two dimensional test problems.
The problems to solve are written in general as 
$\vF(x,y) = \left(f_x(x,y), f_y(x,y)\right)^T = \vzero$
together with the initial condition $\bx_0$ for the iteration.
\begin{enumerate}
 \item Rosenbrock's problem:
       \begin{eqnarray}
	f_x(x,y) &=& 10(y-x^2),\nonumber\\
	f_y(x,y) &=& 1-x,\nonumber\\
	\vx_{0} &=& (1.2, 1)^T.\label{eq:rosenbrock}
       \end{eqnarray}
 \item Freudenstein and Roth's problem:
       \begin{eqnarray}
	f_x(x,y) &=& -13 +x +\left( (5-y)y -2 \right)y,\nonumber\\
	f_y(x,y) &=& -29 +x +\left( (y+1)y -14 \right)y,\nonumber\\
	\vx_{0} &=& (6,3)^T.\label{eq:freudenstein}
       \end{eqnarray}
 \item Powell's badly scaled problem:
       \begin{eqnarray}
	f_x(x,y) &=& 10^4xy -1,\nonumber\\
	f_y(x,y) &=& e^{-x} +e^{-y} -1.0001,\nonumber\\
	\vx_{0} &=& (0,1)^T\ \mathrm{or}\ \vx_{0}=(1,1)^T.\label{eq:powell}
       \end{eqnarray}
 \item Brown's badly scaled problem:
       \begin{eqnarray}
	f_x(x,y) &=& xy^2 -2y +x -10^6,\nonumber\\
	f_y(x,y) &=& x^2y -2x +y -2\times 10^{-6},\nonumber\\
	\vx_{0} &=& (1,1)^T.\label{eq:brown}
       \end{eqnarray}
 \item Beale's problem:
       \begin{eqnarray}
	f_x(x,y) &=& 1.5 -x(1-y),\nonumber\\
	f_y(x,y) &=& 2.25 -x(1-y^2),\nonumber\\
	\vx_{0} &=& (1,1)^T\ \mathrm{or}\ \vx_{0}=(0,2)^T.\label{eq:beale}
       \end{eqnarray}
 \item[A.] Hueso \& Monteiro's problem:
       \begin{eqnarray}
	f_x(x,y) &=& (x-1)^2(x-y),\nonumber\\
	f_y(x,y) &=& (y-2)^5\cos(2x/y),\nonumber\\
	\vx_{0} &=& (1.5,2.5)^T.\label{eq:hueso}
       \end{eqnarray}
 \item[B.] Fujisawa's problem:
       \begin{eqnarray}
	f_x(x,y) &=& x^2 +y^2 -4,\nonumber\\
	f_y(x,y) &=& x^2y -1,\nonumber\\
	\vx_{0} &=& (0,1)^T\ \mathrm{or}\ \vx_{0}=(0,-1)^T.\label{eq:fujisawa}
       \end{eqnarray}
\end{enumerate}

%%%%%%%%%%%%%%%%%%%%%%%%%%%%%%%%%%%%%%%%%%%%%%%%%%%%%%%%%%%%%%%%%%%%%%%%%%%%%%%%%%%%%%%%%%%
\bibliographystyle{elsarticle-num}
\bibliography{nonlinear}

\end{document}